\documentclass[a4paper,11pt]{amsart}

\usepackage[foot]{amsaddr}
\usepackage[T1]{fontenc}
\usepackage[utf8]{inputenc}
\usepackage{palatino}
\usepackage{lipsum}
\usepackage{amsmath, amssymb, mathtools, amsthm}
\usepackage{comment}
\usepackage{esint}
\usepackage[usenames,dvipsnames]{xcolor}
\usepackage{enumerate} 
\usepackage{bbm}

\usepackage[numbers,sort&compress]{natbib}
\usepackage[colorlinks=true]{hyperref}
\hypersetup{urlcolor=blue, citecolor=red} 

\usepackage[title]{appendix}
\usepackage{dsfont}

\DeclarePairedDelimiter{\prt}{(}{)}

\newcommand \commentout[1] {}

\DeclareMathOperator*{\supp}{\operatorname{supp}}

\newcommand{\partialt}[1]{\dfrac{\partial#1}{\partial t}}

\DeclareMathAlphabet{\mathup}{OT1}{\familydefault}{m}{n}
\newcommand{\dx}[1]{\mathop{}\!\mathup{d} #1}
 
\newcommand{\curlyP}{\mathcal{P}}
\newcommand{\curlyB}{\mathcal{B}}

\newcommand{\hn}{{\hat{n}}}
 
\theoremstyle{plain}
\newtheorem{thm}{Theorem}[section]
\newtheorem{lemma}[thm]{Lemma}
\newtheorem{proposition}[thm]{Proposition}
\newtheorem{corollary}[thm]{Corollary}

\theoremstyle{remark}
\newtheorem{remark}[thm]{\bf Remark}
\newtheorem{definition}[thm]{\bf Definition}
\newtheorem{assumptions}[thm]{\bf Assumption}
\newcommand{\ie}{\emph{i.e.}\;}
\newcommand{\cf}{\emph{cf.}\;}

\newcommand{\sign}{\mathrm{sign}}

\newcommand{\R}{\mathbb{R}}
 
\begin{document}

\title[New Lipschitz estimates for nonlinear diffusion]{New Lipschitz estimates and long-time asymptotic behavior for porous medium and fast diffusion equations
}

\author{Noemi David$^*$}
\author{Filippo Santambrogio$^*$}

\address{$^{*}$ Institut Camille Jordan, Université Claude Bernard Lyon 1, 43 bd du 11 novembre 1918, 69622 Villeurbanne Cedex,
France. {Email addresses:} {ndavid@math.univ-lyon1.fr}, {santambrogio@math.univ-lyon1.fr}}

\maketitle
\begin{abstract}
We obtain new estimates for the solution of both the porous medium and the fast diffusion equations by studying the evolution of suitable Lipschitz norms. Our results include instantaneous regularization for all positive times, long-time decay rates of the norms which are sharp and independent of the initial support, and new convergence results to the Barenblatt profile. Moreover, we address nonlinear diffusion equations including quadratic or bounded potentials as well. In the slow diffusion case, our strategy requires exponents close enough to 1, while in the fast diffusion case, our results cover any exponent for which the problem is well-posed and mass-preserving in the whole space.
\end{abstract}

\vskip .3cm
\begin{flushleft}
    \noindent{\makebox[1in]\hrulefill}
\end{flushleft}
\begin{flushright}
    \noindent{\makebox[1in]\hrulefill}
\end{flushright}
\vskip .3cm
 
\section{Introduction} 
The goal of this paper is to provide new estimates on the regularity of the solution $n(t,x)$ of the following nonlinear equation
\begin{equation}\label{eq: density}
   \partialt{n} = \nabla \cdot (n \nabla (p + V)),  \quad\text{in } (0,\infty)\times \Omega, \quad d\geq 2,
\end{equation}
endowed with initial condition $n(0,x)=n_0(x)\geq 0, n_0\in L^1(\Omega)$, where the domain is either the whole space $\Omega=\R^d$ or a convex bounded set, for which the problem is endowed with homogeneous Neumann boundary conditions. As constitutive law of the pressure we take the signed power law
\begin{equation*}
    p=P(n):=\sign(\gamma) n^\gamma,
\end{equation*}
while $V:\R^d\to\R$ is a potential whose regularity will be detailed later on.
Equation~\eqref{eq: density} is a well-known example of a partial differential equation including convective effects and nonlinear diffusion whose theory is nowadays well established. For $\gamma>0$ it is an equation of \textit{porous medium} type, while for $\gamma<0$ it is referred to as \textit{fast diffusion}. In the latter case, we only consider the usual range of exponents $-2/d <\gamma<0$. 
We are interested in proving decay estimates on the quantity
\begin{equation}\label{quantity}
 u(t):=  \max_x |p(t,x)|^b{|\nabla p(t,x) + \nabla V(x)|^2},
\end{equation}
where the coefficient $b \in \R$ will be chosen in an appropriate range. These estimates will provide regularity and long-time asymptotic results.
To this end, we exploit the fact that the evolution of the pressure is described by
\begin{equation}\label{eq: p}
    \partialt{p} =\gamma p \Delta q+ \nabla p \cdot \nabla q, \quad \text{ with } \, q:=p+V.
\end{equation}
Our analysis will focus on three cases: the trivial potential $V=0$, the quadratic potential $V=|x|^2/2$, and a generic potential with bounded derivatives.

\subsection{Historical remarks and recent results}
The porous medium (\textit{resp.} fast diffusion) equation, namely equation~\eqref{eq: density} with $V=0$, is the simplest example of a nonlinear diffusion equation; it can be written as
\begin{equation}\label{eq: pme}
    \partialt{n} = \frac{|\gamma|}{\gamma+1}\Delta n^{ \gamma+1},
\end{equation}
for $\gamma>0$ (\textit{resp.} $-2/d< \gamma<0$). The theory on these equations is nowadays well established, we refer the reader to the monographs \cite{V, V06} for an overview.
In this section, we recall some important properties of the solutions to these equations.

\smallskip
\noindent
\textbf{The porous medium equation.} It is well known that, unlike solutions of the heat equation, solutions of the porous medium equation (PME) can exhibit a free boundary. In fact, since the equation is degenerate rather than uniformly parabolic, the speed of propagation is finite  \cite{OKJ58}. Consequently, if the initial data are compactly supported, the solution remains so for all times, and we may distinguish two regions $\Omega(t):=\{x; \ n(x,t)>0\}$ and $\{x; \ n(x,t)=0\}$, separated by a free interface. For compactly supported initial data, even if continuous, the porous medium equation does not admit a classical solution, since the solution's gradient is discontinuous on the free boundary \cite{kalashnikov}.
However, in small enough neighborhoods of points in which $n(x, t)>0$, solutions are smooth and satisfy the equation in the classical sense. 
Let us now recall a fundamental estimate for the porous medium equation \cite{aronson70, AB79}.
\begin{lemma}[Aronson-B\'enilan estimate]
The pressure $p=n^\gamma$ satisfies
\begin{equation}
    \label{AB}
    \Delta p \geq - \frac{1}{\left(\gamma + \frac 2 d\right)t}, \quad \text{ for all } t>0.
\end{equation}
\end{lemma}
\noindent
This lower bound on the Laplacian of the pressure is used in \cite{AB79} to prove that there exists a unique strong and continuous solution to the Cauchy problem with $L^1$-bounded initial data.

In \cite{CF80}, Caffarelli and Friedman prove that the solution $n(t,x)$ to the porous medium equation is actually H\"older continuous, uniformly in space and time. This result was further developed in \cite{CVW87} where the authors show that, after a certain \textit{waiting time}, the pressure is actually Lipschitz. It is indeed known that due to the finite speed of propagation, singularities may appear in finite time if the support of the initial data contains "holes". While in the one-dimensional case the pressure is Lipschitz for all times $t>0$, \cite{aronson70, kalashnikov}, in higher dimensions the pressure gradient blows up at the \textit{focusing time}, namely when the hole fully closes \cite{AGV98}. In particular, in \cite{CVW87} the authors prove that if $B_{R_0}$ is the smallest ball containing the support of $n_0$, and $t^*:=\inf\{t>0; \overline{B_{R_0}}\subset \Omega(t)\}$ is the focusing time, then $\partial_t p $ and $\nabla p$ are bounded for all $t>t^*$, and the bounds depend on $n_0$ and $t$.

Let us recall that the source solution of the porous medium equation is given by the following self-similar profile, usually referred to as the Barenblatt solution
\begin{equation}
    \label{Barenblatt density pme}
   \begin{aligned} 
   \curlyB_{\mathrm{PME}}(t,x)=t^{-\alpha d}F\prt*{{x}{t^{-\alpha}}}, \quad F(\xi)=\prt*{C - \alpha \frac{|\xi|^2}{2}}^{\frac 1 \gamma}_+, \quad \text{with }
 \alpha= \frac{1}{d \gamma +2},
    \end{aligned}
\end{equation}
where $C>0$ depends on $\gamma, d$, and the mass of the initial data, $M=\| n_0 \|_{1}$.
It is nowadays well established that for $t\to \infty$ the solution of the PME converges to the self-similar profile. The quest for explicit rates of convergence has attracted a lot of attention over the last few decades, see for instance \cite{V2003} and references therein. The solution satisfies the following convergence results
\begin{align*}
    \lim_{t \to \infty} \|n(t)- \curlyB_{\mathrm{PME}}(t)\|_{L^1(\R^d)}=0,\\[0.5em]
    \lim_{t \to \infty} t^{\alpha d} \|n(t)- \curlyB_{\mathrm{PME}}(t)\|_{L^\infty(\R^d)}=0,
\end{align*}
with $\alpha$ defined as in \eqref{Barenblatt density pme}. These rates are optimal for solutions with $L^1(\R^d)$ non-negative initial data.
A related question that has been addressed by a vast literature is whether these rates of convergence can be improved for a different class of initial data. Although it is not the purpose of this paper to review in detail these results, let us mention that entropy methods have been adopted to find better rates for $L^1$-initial data with finite second moment, see for instance \cite{V83, C01, CT00, Ot01}. Moreover, it has been shown that the Fisher information of equation \eqref{eq: density} with $V=|x|^2/2$, namely
\begin{equation*}
    \mathcal{I}(n)= \int_{\R^d} n \left|x + \nabla p\right|^2 \dx x,
\end{equation*} 
decays exponentially
 \begin{equation*}
     \mathcal{I}(n(t))\leq \mathcal{I}(n(t_0)) e^{-\lambda (t-t_0)}, \text{ for } t\geq t_0 >0,
 \end{equation*}
for some $\lambda>0$. This result also holds for the fast diffusion equation for a more restrictive range of exponents \cite{CT00, Blanchet2009}. Let us notice that the functional $\mathcal{I}(n)$ is the integral counterpart of the quantity \eqref{quantity} for $b= 1 /\gamma$, of which we study the asymptotic behavior, \cf Theorem~\ref{thm: attr V}.

Another interesting question is whether it is possible to obtain rates of convergence for the space derivative of the solution of equation \eqref{eq: pme}. Here the main challenge arises due to the presence of solutions with compact support which does not coincide with the one of the self-similar profile. In \cite{LV03}, Lee and V\'azquez show that after a certain time, the pressure is concave and converges to a truncated parabolic profile to all orders of differentiability. The authors assume the initial data to be compactly supported and satisfy a technical non-degeneracy condition.

\smallskip
\noindent
\textbf{Fast diffusion equation.}
Let us now recall some properties of solutions to equation~\eqref{eq: pme} for negative exponents, namely the standard fast diffusion equation (FDE). It is well known that, unlike the porous medium equation, the FDE admits classical solutions if the exponent satisfies $-2/d<\gamma<0$. In particular, for any $n_0\in L^1(\R^d)$ there exists a unique $C^\infty$ solution which is always strictly positive. If $\gamma$ is below the critical threshold $-2/d$, finite extinction phenomena arise and solutions may lose the mass-preservation property. Moreover, for $-2/d<\gamma<0$, the Aronson-B\'enilan estimate \eqref{AB} still holds. 

A self-similar solution with finite mass exists in the range $-2/d<\gamma<0$, and it exhibits so-called \textit{fat tails} for large values of $|x|$ 
\begin{equation*}
   \begin{aligned}
   \curlyB_{\mathrm{FDE}}(t,x)=t^{-\alpha d}F\prt*{{x}{t^{-\alpha}}}, \quad F(\xi)=\prt*{C + \alpha \frac{|\xi|^2}{2}}^{\frac 1 \gamma},\quad \text{with }
 \alpha= \frac{1}{d \gamma +2}.
    \end{aligned}
\end{equation*}
Let us point out that since the exponent belongs to the range $-2/d < \gamma < 0$, we still have $\alpha>0$. Therefore, while $\curlyB_{\mathrm{FDE}}^\gamma$ is now convex, the signed pressure $-\curlyB_{\mathrm{FDE}}^\gamma$ is again concave, though negative.

In the last decades, several results on the convergence of the solution of the FDE to the Barenblatt profile as $t\to \infty$ have been established under different assumptions on the initial data, see for instance \cite{CV2003, KMcC2006, Blanchet2009, V06, BV06}. In the very recent work \cite{BS22}, the authors find a necessary and sufficient condition on the initial data such that the solution convergences to the Barenblatt profile uniformly in \textit{relative error}, or \textit{weighted convergence}, namely
\begin{equation}\label{rel err}
    \lim_{t\to \infty} \left\|\frac{n(t)-\curlyB_{\mathrm{FDE}}(t)}{\curlyB_{\mathrm{FDE}}(t)}\right\|_{L^\infty(\R^d)} =0.
\end{equation}
To achieve this result, they first prove that if the initial data decays for large $|x|$ in a similar way as the source solution (see \eqref{space X} for the definition of this class of data), then the solution can always be bounded from below and above by two Barenblatt profiles of masses $\underline{M}$ and $\overline{M}$, \cf \cite[Theorem 1.1]{BS22}. In particular, under appropriate conditions on $n_0$, for any $t_0>0$, there exist $\underline{\tau},\bar{\tau}>0$, and $\underline{M},\overline{M}>0$ such that
\begin{equation}
    \label{cond: btween B}
    \curlyB_{\mathrm{FDE}}(t-\underline{\tau},x;\underline{M})\leq n(t,x) \leq  \curlyB_{\mathrm{FDE}}(t+\bar{\tau},x;\overline{M}), \quad \forall x\in \R^d, t\geq t_0.
\end{equation} 
A rate of convergence for the relative error has been recently found in \cite[Theorem 4.1]{BDNS22}.

Convergence rates of the solution of the fast diffusion to the Barenblatt profile were found in \cite{Blanchet2009} for any $C^k$-seminorm with $k\in \mathbb{N}$, under the assumption that the initial data is bounded from above and below by two Barenblatt profiles. To this end, the authors use convenient H\"older interpolation inequalities between $L^2(\R^d)$ and $C^{k+1}(\R^d)$. 

\smallskip
\noindent
\textbf{Our contribution compared to the existing results relying on the Bernstein technique.} 
Looking in the literature, it is possible to find some results that use similar techniques to the one we employ in this paper.

An interesting estimate for the flux $\nabla n^{\gamma+1}$ that holds uniformly in space was proved in B\'enilan's notes \cite{Benilannotes} where, for $n_0\in L^\infty(\R^d)$, the author shows
\begin{equation*}
  \big|\nabla n^{\gamma+1}\big|^2\leq \frac{n K_1}{t}, 
\qquad \;
K_1=\frac{2(\gamma+1)^2 \|n_0\|_\infty}{\gamma(1-\gamma^2(d-1))},
\end{equation*}
under the condition $\gamma^2(d-1)<1$. 
Our method recovers this estimate. Indeed, the above inequality can be rewritten in terms of the quantity $u(t)$ in \eqref{quantity} for $b=1/\gamma$, and gives
\begin{equation*}
u(t)= \max_x n \big|\nabla n^{\gamma}\big|^2\leq \frac{K_2}{t}, 
\qquad \;
K_2=\frac{2\gamma \|n_0\|_\infty}{1-\gamma^2(d-1)},
\end{equation*}
which is the same estimate we provide in Theorem~\ref{thm: V=0}. As we discuss later in the paper, the choice $b=1/\gamma$ is natural since it is the one which minimizes the coefficient $c_0$ appearing in \eqref{defic_0}, and making it as negative as possible. It is also the choice of $b$ which provides an estimate on $n|\nabla p|^2$, in analogy with the role played by the Fisher information. The method applied in \cite{Benilannotes} relies on a modified Bernstein technique, which is, in its essence, analogous to the strategy used in this paper to study \eqref{quantity}. 

The same strategy of \cite{Benilannotes} was also adapted in \cite{BGIL16} to provide estimates on the fast diffusion equation with critical zero-order absorption. In this work, the authors study the Lipschitz norm of the square root of the pressure, namely $n^{\gamma/2}$, which corresponds in our setting to the choice $b= -1$. 
This choice of exponent is natural in that, in the fast diffusion case, the optimal Lipschitz regularity is expected to be satisfied by $\sqrt{p}$ since the pressure behaves like $C(1+|x|^2)$. Once again, our strategy covers this case. 

In \cite{CJM15} the authors prove a priori gradient bounds for the solution itself, for a large family of nonlinear parabolic equations. This includes equation~\eqref{eq: pme} for $0\leq\gamma\leq 4/(d+3)$ and it is proven that the Lipschitz regularity of the solution $n(t,x)$ is preserved. Our main result also covers this case, since $\|\nabla n\|_\infty$ corresponds to the choice $b= 2/\gamma-2$ and the conditions for our theory to apply exactly require $0\leq\gamma\leq 4/(d+3)$. Moreover, we not only recover the same result as in \cite{CJM15} for the porous medium equation, but, when the inequality $0<\gamma< 4/(d+3)$ is strict, we show instantaneous regularization of the Lipschitz norm for solutions with $L^1$ initial data.

The main novelties of our paper are two. The first one concerns the choice of the exponent $b$. Indeed, to this day the existing literature has mainly been focusing only on particular choices of the exponent $b$ in \eqref{quantity}. On the contrary, we propose a comprehensive study where $b$ and $\gamma$ are considered as parameters and we look at the conditions that guarantee the decay of the corresponding quantities. 
Similar results for more general exponents have been obtained in \cite{HQZ17}, where the authors use probabilistic methods, based in particular on martingale integration, to establish gradient estimates on the solution of the porous medium and fast diffusion equations. The results in \cite{HQZ17} are probably the closest to our analysis, but the technique is different, and their results are in general expressed in terms of local quantities which makes it difficult to see their expression in terms of natural decay.

The second new aspect concerns the insertion of drift effects. So far, the Bernstein technique has not been applied to the convective-~(nonlinear)diffusion case, equation~\eqref{eq: density}. Yet, considering non-trivial convective effects, as discussed at the end of this section, allows us to obtain new insightful results on the convergence of the solution to the self-similar profile. As a consequence, the possibility to easily consider a drift should not be seen as a mere technical improvement, but as a core feature of the theory.

To summarize, our current contribution is to provide new results on the study of suitable Lipschitz bounds for solutions of nonlinear diffusion equations including a drift term, namely estimating the quantity defined in \eqref{quantity}. This strategy has the advantage of working both for the PME and the FDE in a unified way, and it is essentially independent of assumptions on the initial data. It yields new results on regularity (valid for $t>0$ and not only after some focusing time) and asymptotic behavior. However, for $\gamma>0$ our method only works for very small $\gamma$, namely when the diffusion is almost linear. Under mild assumptions on the potential $V$, we provide results on equation~\eqref{eq: density} at least for smooth and suitably decaying solutions, \cf Proposition~\ref{prop: main}. For the most standard cases $V=0$ and $V=|x|^2/2$, our results actually hold for general solutions and general initial data. In particular, for the standard equation \eqref{eq: pme} we consider as initial data any  $L^1$-non-negative function, \cf Theorems~\ref{thm: V=0} and \ref{thm: attr V}. 
The main novelty introduced in this paper is to consider equations that also include a drift term. In particular, in the quadratic case, we exploit the equivalence between the nonlinear Fokker-Plank equation and the standard equation (by means of the time-dependent scaling discussed in the following section) hence providing a new method to infer weighted convergence results of the pressure gradient to the self-similar profile, \cf Theorem~\ref{thm: convergence grad}. In the fast diffusion case, when \eqref{rel err} holds, such weight - which actually depends on $n$ - can be replaced by an explicit function of time and space, leading to a new convergence rate in the $C^1$-seminorm.

The estimates we present in this paper are mainly applied to the study of the asymptotic behavior or the instantaneous regularization of the solutions of the PME or FDE. Yet, one of the main interests of the present paper is, in our opinion, that it shows that, despite the huge literature existing on these equations, it is still possible to find new and simple estimates with quite elementary techniques.

\subsection{Preliminaries and assumptions.}\textbf{ }\\ 

\noindent
\textbf{Time-dependent scaling.} A fundamental remark that has been extensively used in the literature to study the properties of the standard PME and FDE consists in observing that solutions to equation~\eqref{eq: pme} can actually be seen as solutions of a nonlinear Fokker-Planck equation with quadratic potential $V=|x|^2/2$ through the following time-dependent change of variables
\begin{equation}\label{eq: change of v}
\begin{aligned}
    \hn (t,x) := \varphi(t)^d n (\psi(t), \varphi(t)x), \qquad
        \varphi(t)=e^t, \;
        \psi(t)=e^{(d\gamma+2)t}. 
\end{aligned}        
\end{equation}
In fact, if $n(t,x)$ is a solution of equation~\eqref{eq: pme}, $\hat n (t,x)$ satisfies
\begin{equation}\label{eq hat n}
    \partialt{\hn} = \frac{|\gamma|}{\gamma+1}\Delta \hn^{\gamma+1} + \nabla \cdot (\hn x).
\end{equation}
Unlike the drift-less case, equation~\eqref{eq hat n} has a unique compactly supported stationary state, which coincides with the Barenblatt profile evaluated at an appropriate time $t=t(d,\gamma)$. This property allows us to infer long-time behavior results on the porous medium/fast diffusion equation from the long-time behavior of solutions of this Fokker-Planck equation.

\smallskip
\noindent
\textbf{Tail behavior of the FDE solution.}
We now review some particular properties of the FDE that will be used throughout the paper.  
It is known that if $\gamma$ belongs to the range $-2/d<\gamma<0$, the solution exhibits polynomial tails. In \cite[Theorem 1.1]{BS22}, the authors give a necessary and sufficient condition for the solution to satisfy \eqref{cond: btween B}, namely to be bounded from below and above by two Barenblatt profiles, which is $n_0\in \mathcal{X}\setminus \{0\}$, where
\begin{equation}
    \label{space X}
    \mathcal{X} := \{u \in L^1(\R^d), u \geq 0, |u|_\mathcal{X} <\infty\}, \text{ with } |u|_\mathcal{X} := \sup_{R>0}R^{-\frac{2}{\gamma}-d} \int_{B_R^c}|u|\dx x < \infty.
\end{equation}
Moreover, as shown in the proof of \cite[Theorem~4]{Blanchet2009}, the gradient of the solution can be bounded uniformly by $|x|^{2/\gamma -1}$. 
We collect these properties as follows: for all $t_0$, there exist $k(t), K(t), c(t)\in L^{\infty}_{loc}(0,\infty)$ and $R>0$ such that $\forall |x|>R, t\geq t_0$
\begin{equation}\label{cond: tail behavior}
k(t) (1+|x|^2)^{1/\gamma} \le n(t,x)\leq K(t) (1+|x|^2)^{1/\gamma}, \quad |\nabla  n(t,x)|\leq c(t) |x|^{2/\gamma-1}.
\end{equation}
This tail behavior will be essential for us in order to rigorously justify the formal computations. 
However, this behavior is only known to hold in the drift-less case \eqref{eq: pme} and, through the change of variables \eqref{eq: change of v}, for equation~\eqref{eq hat n}. In order to provide a rigorous justification for the more general case, namely equation~\eqref{eq: density} for a generic potential $V$, we will impose \eqref{cond: tail behavior} as an assumption on the solution, see Definition~\ref{definition}. In the trivial and quadratic potential cases, this condition will be removed later on using approximation arguments.

\smallskip
\noindent
\textbf{$\boldsymbol{L^\infty}$-regularization of $\boldsymbol{n}$.}
An important consequence of the \textit{semi}-~\textit{subharmonicity} of the pressure given by the fundamental estimate \eqref{AB} is to provide a local bound on $\|n(t)\|_\infty$. In the porous medium case, this translates into a bound on $\|p(t)\|_\infty$, while for the fast diffusion, we infer a uniform (in space) lower bound.

\begin{lemma}\label{lemma: bound on n(t)}
    Let $n(t,x)$ be the solution of equation~\eqref{eq: pme} with $|\gamma|<1$. There exists a positive constant C such that the $L^\infty$-norm of $n(t)$ satisfies 
    \begin{equation}
        \label{eq: bound on n(t)}
        \|n(t)\|_\infty \leq C t^{-d\alpha},
    \end{equation}
which is equivalent to  the following bounds on the pressure   
\begin{itemize}
    \item $\max_x |p(t)|\leq C  t^{-d\gamma\alpha}$, for $\gamma>0$,
    \item $\min_x |p(t)|\geq C  t^{-d\gamma\alpha}$, for $\gamma<0$.
\end{itemize}
\end{lemma}

\begin{proof}
Let $\gamma>0$, and let $\Bar{x}$ be a point in the support of $n(t)$. We denote $B_r(\Bar{x})$ the ball with radius $r>0$ centered at $\Bar{x}$.
Thanks to the Aronson-B\'enilan estimate \eqref{AB}, we know that the function $f(x):= p(x) + \frac{K}{2d} |x-\Bar{x}|^2$ is subharmonic for all $\bar{x}\in \R^d$, where $K=\frac{1}{(\gamma +2/d)t}$. Therefore,
from
\begin{equation*}
    f(\bar{x}) \leq  \fint_{B_r(\Bar{x})} f(x)\dx x,
\end{equation*}
we find
\begin{equation*}
    \fint_{B_r(\Bar{x})} p \dx{x} \geq p(\Bar{x})- \frac{K r^2}{2(d+2)}.
\end{equation*}
Let us choose a radius $R>0$ such that $R^2= c p(\Bar{x}) t$, with $c>0$ small enough. Then, for all $0<r\leq R$ we have
\begin{equation*}
    \fint_{B_r(\Bar{x})} p \dx{x} \geq  C p(\Bar{x}),
\end{equation*}
where from now on $C>0$ denotes a constant that may change value from line to line.
Since $p=n^\gamma$ and we are considering a range in which $\gamma<1$, by Jensen's inequality, we have
\begin{equation*}
    \fint_{B_r(\Bar{x})}n \dx{x}= \fint_{B_r(\Bar{x})} p^{1/\gamma} \dx{x} \geq \left(\fint_{B_r(\Bar{x})} p\dx{x}\right)^{1/\gamma} \geq C p(\Bar{x})^{1/\gamma}.
\end{equation*}
Using the fact that $n$ has constant mass at all times, integrating between $0$ and $R$ we find
\begin{equation*}
   M \geq \int_{B_R(\Bar{x})} n \dx{x} \geq R^d p(\Bar{x})^{1/\gamma}.
\end{equation*}
By definition $R^2= c p(\Bar{x}) t$, and we may finally establish the following bound
\begin{equation*}
    p(\Bar{x}) \leq C t^{-\frac{d \gamma}{d \gamma +2}}.
\end{equation*}
Since $\Bar{x}$ was chosen arbitrarily, we conclude that $p\in L^\infty_{loc}(0,\infty; L^\infty(\R^d))$ and the same upper-bound hols for the $L^\infty$-norm of the pressure, while equation~\eqref{eq: bound on n(t)} holds for $n(t)$.

For $-2/d<\gamma<0$ the Aronson-B\'enilan estimate still holds for $p=-n^\gamma$. Therefore, we may argue in the same way choosing $R^2=c|p(\Bar{x})|t$ to obtain
\begin{equation*}
    \fint_{B_r(\Bar{x})} p \dx{x} \geq   C p(\bar{x}).
\end{equation*}
Since $\gamma<0$, the function $s\mapsto s^{1/\gamma}$ is convex and decreasing on $s>0$, so that we can use Jensen's inequality and obtain, for $p=-n^\gamma$:
\begin{equation*}
    \fint_{B_r(\Bar{x})}n \dx{x}= \fint_{B_r(\Bar{x})} (-p)^{1/\gamma} \dx{x} \geq \left(- \fint_{B_r(\Bar{x})} p\dx{x}\right)^{1/\gamma} \geq C |p(\bar{x})|^{1/\gamma}.
\end{equation*}
Using again that the mass is preserved, we find
\begin{equation*}
   M \geq \int_{B_R(\Bar{x})} n \dx{x} \geq R^d |p(\bar{x})|^{1/\gamma},
\end{equation*}
from which we conclude
\begin{equation*}
   \min_x|p(x)|\geq C t^{-\frac{d\gamma}{d \gamma+2}},
\end{equation*}
which is the claim.
\end{proof}

\smallskip
\noindent
\textbf{Assumptions.}
We now state the assumptions that we will alternatively impose on the potential $V$ and the coefficients $b, \gamma$.
\begin{assumptions}\label{assum: bdd V}
The potential $V$ satisfies $|\nabla V|, D^2 V \in L^\infty(\R^d).$
\end{assumptions}
\begin{assumptions}\label{assum: attr V}
    The potential $V$ is the quadratic one: $V(x)=|x|^2/2$.
\end{assumptions}
\begin{assumptions}\label{assum: coeff bdd V}
For $\gamma,b>0$, \textit{resp.} $\gamma,b <0$, we assume
\begin{itemize}
    \item[\textit{(i)}] $\gamma\leq \min\prt*{\dfrac{1}{\sqrt{d}}, \dfrac{2}{d},\dfrac 1 2}$, \textit{resp.} $|\gamma|<\min\prt*{\dfrac{2}{d}, \dfrac{4}{3+d}}$,\\[0.3em]
    \item[\textit{(ii)}] ${1-\sqrt{1-\gamma^2(d-1)}} < \gamma b < {1+\sqrt{1-\gamma^2(d-1)}}$,\\[0.3em]
    \item[\textit{(iii)}]  $\gamma \leq \gamma b \leq 1-\gamma$, \textit{resp.} $|\gamma|<\gamma b\leq\min\prt*{1+|\gamma|, 2 |\gamma|}$.
\end{itemize} 
\end{assumptions}

\begin{assumptions}\label{assum: coeff attr V}
For $\gamma,b>0$, \textit{resp.} $\gamma, b<0$, we assume
\begin{itemize}
    \item[\textit{(i)}] $\gamma<\min\left(\dfrac{1}{\sqrt{d}},\dfrac 2 d\right)$, \textit{resp.} $|\gamma|<\dfrac 2 d$,\\[0.3em]
    \item[\textit{(ii)}] ${1-\sqrt{1-\gamma^2(d-1)}} \leq \gamma b \leq {1+\sqrt{1-\gamma^2(d-1)}}$,\\[0.3em]
    \item[\textit{(iii)}] $\gamma b<\min\prt*{1-\gamma, \dfrac{2}{d}}$, \textit{resp.}  $|\gamma|< \gamma b<\min\prt*{1+|\gamma|, \dfrac{2}{d}}$.
\end{itemize} 
\end{assumptions} 

\begin{assumptions}
    \label{assum: coeff V=0}
For $\gamma,b>0$, \textit{resp.} $\gamma<0,b<-1$, we assume
\begin{itemize}
   \item[\textit{(i)}] $\gamma<\dfrac{1}{\sqrt{d-1}}$, \textit{resp.} $|\gamma|<\dfrac 2 d$,\\[0.3em]
   \item[\textit{(ii)}] $1-\sqrt{1-\gamma^2(d-1)}< \gamma b < 1+\sqrt{1-\gamma^2(d-1)}$.
\end{itemize} 
\end{assumptions}

\noindent
\begin{remark}
For assumptions (\ref{assum: coeff bdd V}, \ref{assum: coeff attr V}) condition \textit{(i)} ensure that there exists $b$ that satisfies both \textit{(ii)} and \textit{(iii)}. For the sake of completeness, we here report the proof of this statement.
For assumption~\ref{assum: coeff bdd V}, $\gamma>0$, condition \textit{(i)} ensures that:
\begin{align*}
\gamma < 1/\sqrt{d} &\Rightarrow   {1-\gamma}> {1-\sqrt{1-\gamma^2(d-1)}},\\
 \gamma < 2/d  &\Rightarrow  {\gamma}< {1+\sqrt{1-\gamma^2(d-1)}},\\
  \gamma< 1/2 &\Rightarrow {\gamma}< {1-\gamma}.
\end{align*}
This guarantees that  under condition \textit{(i)} we have
$$\max\{\gamma,1-\sqrt{1-\gamma^2(d-1)}\}<\min\{1-\gamma,1+\sqrt{1-\gamma^2(d-1)}\}.$$
For $\gamma<0$, condition \textit{(i)} gives
\begin{align*}
   |\gamma| < 2/d &\Rightarrow  {|\gamma|}< {1+\sqrt{1-\gamma^2(d-1)}},\\
    |\gamma|<4/(3+d) &\Rightarrow {2|\gamma|}> {1-\sqrt{1-\gamma^2(d-1)}},
\end{align*}
which guarantees in this case
$$\max\{|\gamma|,1-\sqrt{1-\gamma^2(d-1)}\}<\min\{1+|\gamma|, 2|\gamma|,1+\sqrt{1-\gamma^2(d-1)}\}.$$
Therefore, in both cases, there exists at least one $b\in \R$, $\gamma b>0$, that satisfies both \textit{(ii)} and \textit{(iii)}.

For assumption~\ref{assum: coeff attr V}, $\gamma>0$, we first note that we have
$$ 2/d > 1 - \sqrt{1-\gamma^2 (d-1)},$$
and then, for $\gamma>0$
$$
\gamma < 1/\sqrt{d} \Rightarrow   {1-\gamma}> {1-\sqrt{1-\gamma^2(d-1)}},$$
while for $\gamma<0$,
\begin{align*}
   |\gamma| < 2/d &\Rightarrow  {|\gamma|}< {1+\sqrt{1-\gamma^2(d-1)}},
\end{align*}
and the same conclusion holds in the same way.
\end{remark} 
 
\subsection{Main results}

We may now state the main results of the paper. 
First of all, let us give the definition of solutions that allows us to justify our computations. We later prove that in the most relevant cases, these assumptions can be removed by approximation arguments.
\begin{definition}\label{definition}
Let $n(t,x)$ be a strong solution of equation~\eqref{eq: density}. We say that $n(t,x)$ is a \textit{well-behaved} solution if
\begin{itemize}
    \item \textit{(for $\gamma>0$)} $n$ is continuous on $(0,\infty)\times\R^d$, and for each $t>0$ the the support of $n(t)$ is compact and contained in a ball $B(0,R(t))$ where $R$ is locally bounded on $(0,\infty)$; moreover, $n\in C^{\infty}(\{(t,x); \ n(t,x)>0\})$, and $p \in L_{loc}^\infty((0,\infty); W^{1,\infty}(\R^d))$,
    \item  \textit{(for $\gamma<0$)} $n\in C^{\infty}((0,\infty)\times\R^d)$, $n>0$ on $(0,\infty)\times\R^d$, and $n$ satisfies \eqref{cond: tail behavior}.
    \end{itemize}
\end{definition} 

\noindent

\begin{remark}
 In order to justify the formal computation carried on in the proof of Proposition~\ref{prop: main} we assume to deal with a solution satisfying the conditions of definition~\ref{definition}. 
As will be discussed in Section~\ref{sec: proof main}, for the cases $V=0$, $V=|x|^2/2$ (which are equivalent one to the other thanks to the time-dependent scaling \eqref{eq: change of v}), a solution that satisfies definition~\ref{definition} exists for a large class of initial data which is, in fact, dense in $L^1(\R^d)$. Therefore, we will show that the same conclusions hold assuming only $n_0\in L^1(\R^d)$ upon using an approximation argument and proving the convergence of the approximating sequence.
\end{remark}

\smallskip
\noindent
\textbf{Decay estimates on Lipschitz norms.}

\begin{proposition}\label{prop: main}
Let $n(t,x)$ be a solution to \eqref{eq: density} that satisfies Definition~\ref{definition}. Then, 
\begin{itemize}
    \item \textbf{Generic $\boldsymbol{V}$:} under assumptions~\eqref{assum: bdd V} and~\eqref{assum: coeff bdd V}, if $n \in L^\infty((0,\infty)\times \R^d)$, there exist positive constants $C_1, C_2$ such that
    \begin{equation*}
        \max_x |p(t)|^b |\nabla q(t)|^2 \leq \max\prt*{C_1,C_2 t^{-1}},  \text{ for all } t>0. 
    \end{equation*}
    If moreover, $\sup_x |p_0|^b |\nabla p_0 + \nabla V|^2<\infty$, we have
     \begin{equation*}
        \max_x |p(t)|^b |\nabla q(t)|^2 \leq C,  \text{ for all } t>0. 
    \end{equation*}
    \item \textbf{Quadratic} $\boldsymbol{V}\textbf{:}$ under assumptions~\eqref{assum: attr V} and~\eqref{assum: coeff attr V}, and $C_0:=\sup_x |p_0|^b |\nabla p_0 + x|^2<\infty$, there exists a positive constant $C$ such that
    \begin{equation*}
    \max_x  |p(t)|^b |\nabla q(t)|^2 \leq C_0 e^{-C t}, \text{ for all } t>0. 
\end{equation*}
\item \textbf{Trivial} $\boldsymbol{V}$\textbf{:} assuming $V=0$, under assumption~\eqref{assum: coeff V=0} there exists a positive constant $C$ such that
\begin{equation*}
  \max_x  |p(t)|^b |\nabla p(t)|^2 \leq C t^{-1-\gamma d (b+1) \alpha}, \text{ for all } t>0,
\end{equation*}
where $\alpha$ is given in \eqref{Barenblatt density pme}.
Moreover, the above exponent is sharp.

\noindent
Assuming $n_0\in L^1(\R^d)\cap L^\infty(\R^d)$, there exists a positive constant $C$ such that
\begin{equation*}
  \max_x  |p(t)|^b |\nabla p(t)|^2 \leq C t^{-1}, \text{ for all } t>0. 
\end{equation*}
\end{itemize}
\end{proposition} 

\noindent

For both trivial and quadratic potentials, the same estimates hold for a much larger class of solutions.
\begin{thm}[Trivial potential]\label{thm: V=0}
Let $V=0$, and $n(t,x)$ be the solution of equation~\eqref{eq: density} with initial data $n_0 \in L^1(\R^d)$. Under assumptions (\ref{assum: coeff V=0}), there exists a positive constant $C$ such that
\begin{equation*} 
  \max_x  |p(t)|^b |\nabla p(t)|^2 \leq C  t^{-1-\gamma d (b+1)\alpha}, \text{ for all } t>0.
\end{equation*}
Moreover, this exponent is sharp in that equality holds for the Barenblatt solution.

\noindent
Finally, assuming $n_0\in L^1(\R^d)\cap L^\infty(\R^d)$, there exists a positive constant $C$ such that
\begin{equation*} 
  \max_x  |p(t)|^b |\nabla p(t)|^2 \leq C t^{-1}, \text{ for all } t>0. 
\end{equation*}
\end{thm} 

\begin{remark}
One can actually show that given the $L^\infty$-norm of the initial data, $N:=\|n_0\|_\infty$, $t^{-1}$ is the sharpest polynomial decay of the quantity $\max_x |p(t)|^b |\nabla p(t)|^2$. This property can be easily verified by considering the solution whose initial data is the Barenblatt profile at some given time $t_0$, with $L^\infty$-norm equal to $N$. Then, computing the quantity $|p(t, x)|^b |\nabla p(t, x)|^2$ in $|x|^2=t_0$ and $t=\varepsilon t_0$ one can show that its maximum satisfies $u(t)=\mathcal{O}(1/t)$.
\end{remark}

\begin{thm}[Quadratic potential]\label{thm: attr V}
Let $V=|x|^2/2$, and $n(t,x)$ be the solution of equation~\eqref{eq: density} with initial data $n_0\in L^1(\R^d)$. If $C_0:=\sup_x |p_0|^b |\nabla p_0 + x|^2<\infty$, under assumptions \eqref{assum: attr V} and \eqref{assum: coeff attr V} the following holds
    \begin{equation*} 
    \max_x  |p(t)|^b |\nabla q(t)|^2 \leq C_0 e^{- C t}, \text{ for all } t>0,
    \end{equation*}
    with $C = 1-\gamma b d/2>0.$
\end{thm}
\begin{remark}
Under a weaker assumption on the initial data, we may state that the above result holds away from $t=0$. If we call  $\hat{n}$ a solution of equation \eqref{eq: density} (and $\hat p$ the corresponding pressure)  with $V=|x|^2/2$ then, using the change of variables in \eqref{eq: change of v} the corresponding $n$ solves equation~\eqref{eq: pme}. This means that we have instantaneous regularization for $n$ and hence also for $\hat n$. In particular, for any $\tau>0$, $\max_x |p(\tau)|^b |\nabla p(\tau)|^2\leq C(\tau)$. Thus, assuming only $\sup_x |\hat{p}(\tau)|^b |x|^2<\infty$, we obtain that Theorem~\ref{thm: attr V} holds for $\hat{n}(t,x)$ for all $t>\tau$.
\end{remark}

\begin{corollary}[Optimal results for $\sqrt{p}$]\label{cor: sqrt}
    Let $V=0$, $\gamma<0$ and $n(t,x)$ be the solution of equation~\eqref{eq: density} with initial data $n_0\in L^1(\R^d)$. Then, the pressure $p(t,x)=-n(t,x)^\gamma$ satisfies the following estimate
    \begin{equation*}
        \max_x \left|\nabla \sqrt{p(t,x)}\right|^2 \leq \frac{2\alpha}{t},
    \end{equation*}
    where $\alpha=1/(d\gamma+2)$.
Moreover, the universal constant $2\alpha$ is optimal.
\end{corollary}

\begin{corollary}[Lipschitz control of $n$] \label{cor: lip n}
 Let $V=0$, $0<\gamma\leq 4/(d+3)$ and $n$ be the solution of equation~\eqref{eq: density} with initial data $n_0\in L^1(\R^d)$. Then, the quantity 
 $$ \max_x |\nabla n(t,x)|^2$$
 decreases in time. Moreover, if the inequality $\gamma< 4/(d+3)$ is strict, then 
 there exists a positive constant $C$ such that
    \begin{equation*}
        \max_x |\nabla n(t,x)|^2 \leq C t^{-1-\alpha d(2-\gamma)},
    \end{equation*}
where $\alpha=1/(d\gamma+2)$. 
\end{corollary}

\noindent
\textbf{Long-time convergence to the Barenblatt profile.}
As already mentioned, through the change of variables \eqref{eq: change of v}, results for the solution to the drift-less equation can be translated into results for the Fokker-Planck equation and viceversa. However, let us remark that the results stated above, namely Theorem~\ref{thm: V=0} and Theorem~\ref{thm: attr V}, actually provide two different information, since the quantity for which we find decay rates is not the same for the two equations. Therefore, the latter result gives additional insights into the asymptotics of the solution to the standard equation, in particular, it provides convergence results of its gradient to the gradient of the self-similar profile.
Applying Theorem~\ref{thm: attr V} to the rescaled solution allows us to state the following result.
\begin{thm}[Weighted convergence to the Barenblatt gradient]\label{thm: convergence grad}
Let $n(t,x)$ be the solution of equation~\eqref{eq: pme}. If $C_0:=\sup_x |p_0|^b |\nabla p_0 + x|^2< \infty$, under assumptions \eqref{assum: coeff attr V} we have
\begin{equation}\label{eq: thm rate grad}
 \max_x  |p(t,x)|^{b}\left|\nabla p(t,x)  + xt^{-1}\right|^2 \leq C_0 t^{\beta} t^{-\frac{C}{d\gamma +2}}, \text{ for all } t>0,
\end{equation}
with  $C=1-\gamma bd/2$ and  
\begin{equation*}
\beta:=-\alpha\gamma d b -2 + 2\alpha.
\end{equation*}
\end{thm}

The above inequality provides a convergence result of the pressure gradient towards the gradient of the source solution, although weighted by $|p|^b$. Let us recall that the pressure of the Barenblatt profile is 
\begin{equation}
    \label{Barenblatt pressure}
    \curlyP(t,x):= t^{-\alpha\gamma d}\prt*{C - \sign(\gamma)\alpha \frac{|x|^2}{2} t^{-2\alpha}}_+.
\end{equation}  
and thus its gradient is proportional to $-x t^{-1}$. Moreover, the exponent $\beta$ is exactly the sharp exponent of Theorem~\ref{thm: V=0}, which means $|\curlyP|^b|\nabla \curlyP|^2 \simeq t^\beta$ in cylinders of the form $\{(t,x); \ k|x|=t^\alpha\}$.
Therefore, in the fast diffusion case, assuming to be in a regime in which the solution $n(t,x)$ can actually be approximated by a Barenblatt profile, we may remove the weight from \eqref{eq: thm rate grad} and infer a new convergence result of the pressure gradient.

\smallskip
\noindent
\textbf{Structure of the paper.}
In Section~\ref{sec: proof main}, we study the decay of the Lipschitz-like norm defined in \eqref{quantity} and prove Proposition~\ref{prop: main}. Then, using an approximation argument we prove Theorems~\ref{thm: V=0} and \ref{thm: attr V}, namely the results for any $L^1$-bounded initial data in the drift-less and quadratic potential cases. Section~\ref{sec: asymptotic results} deals with the proof of Theorem~\ref{thm: convergence grad} from which we infer the long-time asymptotic behavior of the solution to the fast diffusion equation. Finally, in Section~\ref{sec: bounded} we show that the same asymptotic results obtained in the whole space also hold for the problem set in a convex bounded domain with homogeneous Neumann boundary conditions and also discuss the homogeneous Dirichlet case.

\section{Decay estimates on Lipschitz norms}\label{sec: proof main}
First, we show the main formal computation, from which we will deduce the decay results of Proposition~\ref{prop: main} which will be fully justified later on in this section.
Then, we prove Theorems~\ref{thm: V=0} and \ref{thm: attr V} through a regularization argument.

Given $t_0>0$, let us assume that there exists $x_0\in \R^d$ such that 
\begin{equation}
    \label{max}
    \max_{x} e^{\alpha(p(t_0,x))}{|\nabla q(t_0,x)|^2}= e^{\alpha(p(t_0,x_0))}{|\nabla q(t_0,x_0)|^2},
\end{equation}
where $\alpha=\alpha(p)$ is a continuously differentiable function of the pressure, which will be later chosen as $\alpha(p)=b\ln |p|$, $b\in \R$. We denote
\begin{equation*}
   v(t):= e^{\alpha(p(t,x_0))}\frac{ |\nabla q(t,x_0)|^2}{2},
\end{equation*}
We now compute $\left(\frac{d}{dt} v(t,x_0)\right)_{|t=t_0}.$ For the sake of simplicity, from now on we omit the dependency upon $t_0$ and $x_0$, and we write $\alpha,\alpha'$, and $\alpha''$ for $\alpha(p),\alpha'(p)$, and $\alpha''(p)$.
Since $x_0$ is a maximum point, the following conditions hold
\begin{align*}
    \nabla \left(e^{\alpha(p)}\frac{|\nabla q|^2}{2}\right) =0, \qquad
    \Delta \left(e^{\alpha(p)}\frac{|\nabla q|^2}{2}\right)\leq 0.
\end{align*}
The first condition gives
\begin{align}\label{D2q}
    D^2 q \nabla q = -\frac{\alpha'}{2} |\nabla q|^2 \nabla p,
\end{align}
while from the second we infer
\begin{align*}
\Delta &\left(e^{\alpha}\frac{|\nabla q|^2}{2}\right)\\[0.5em]
  \qquad&= e^\alpha \left( \frac{\alpha''+|\alpha'|^2}{2} |\nabla q|^2 |\nabla p|^2 +2 \alpha' \nabla p D^2 q \nabla q + \alpha' \frac{|\nabla q|^2}{2}\Delta p + \Delta \left(\frac{|\nabla q|^2}{2}\right) \right)\\[0.5em]
 \qquad& \leq 0.
\end{align*}
Using \eqref{D2q}, we find
\begin{equation}\label{deltaleq}
  \Delta \left(\frac{|\nabla q|^2}{2}\right)\leq \frac{|\nabla q|^2|\nabla p|^2}{2} \left(|\alpha'|^2 - \alpha''\right) - \alpha' \frac{|\nabla q|^2}{2}\Delta p.
\end{equation}
We compute the time derivative using equation \eqref{eq: p}
\begin{align*}
    \partialt{}\left( e^\alpha \frac{|\nabla q|^2}{2}\right) &= e^\alpha \left(\alpha' \frac{|\nabla q|^2}{2} \left(\gamma p\Delta q + \nabla p \cdot\nabla q\right)+\nabla q \cdot \nabla \left(\gamma p \Delta q + \nabla p \cdot\nabla q\right)
    \right)\\[0.6em]
    &=e^\alpha \left(\gamma p \frac{\alpha'}{2}|\nabla q|^2 \Delta q + \frac{\alpha'}{2}|\nabla q|^2 \nabla p \cdot \nabla q + \gamma \nabla q \cdot \nabla p \Delta q \right.\\[0.2em]
    & \left.\qquad\qquad + \gamma p \nabla q \cdot \nabla \Delta q  + \nabla q D^2 p \nabla q + \nabla q D^2 q \nabla p
    \frac{}{}\right).
\end{align*}
Using \eqref{D2q} and the equality $2\nabla q \cdot \nabla (\Delta q)= \Delta(|\nabla q|^2)-2 |D^2 q|^2$, where $|D^2 q|^2:= \sum_{i,j=1}^d (\partial_{i,j} q)^2$, we obtain
\begin{align*}
    \partialt{}&\left( e^\alpha \frac{|\nabla q|^2}{2}\right) \\[0.6em]
    &=e^\alpha \left(\gamma p \frac{\alpha'}{2}|\nabla q|^2 \Delta q + \frac{\alpha'}{2}|\nabla q|^2 \nabla p \cdot \nabla q + \gamma \nabla q \cdot \nabla p \Delta q +\gamma p \Delta \prt*{\frac{|\nabla q|^2}{2}} \right.\\[0.2em]
    & \left.\qquad\qquad - \gamma p |D^2 q|^2 - \frac{\alpha'}{2}|\nabla q|^2 \nabla q \cdot \nabla p - \nabla q D^2 V  \nabla q - \frac{\alpha'}{2}|\nabla q|^2 |\nabla p|^2
    \frac{}{}\right)\\[0.6em]
    &\leq e^\alpha \left(\gamma p \frac{\alpha'}{2}|\nabla q|^2 \Delta q + \gamma \nabla q \cdot\nabla p \Delta q + \gamma p \frac{|\nabla q|^2|\nabla p|^2}{2}(|\alpha'|^2 - \alpha'') \right.\\[0.2em]
    & \left.\qquad\qquad  - \gamma p \frac{\alpha'}{2}|\nabla q|^2\Delta p - \gamma p |D^2 q|^2 - \nabla q D^2V \nabla q - \frac{\alpha'}{2}|\nabla q|^2 |\nabla p|^2
    \right)\\[0.6em]
    &= e^\alpha \left(\gamma p \frac{\alpha'}{2}|\nabla q|^2 \Delta V + \gamma p \frac{|\nabla q|^2|\nabla p|^2}{2}(|\alpha'|^2 - \alpha'') \right.\\[0.2em]
    & \left.\qquad\qquad + \underbrace{\gamma \nabla q \cdot\nabla p \Delta q  - \gamma p |D^2 q|^2}_{\mathcal{A}} - \nabla q D^2V \nabla q - \frac{\alpha'}{2}|\nabla q|^2 |\nabla p|^2
    \right),
\end{align*}
where in the above inequality we used \eqref{deltaleq} and the fact that $\gamma p >0$.

Let us now treat the term $\mathcal{A}$. We choose some coordinates where the first vector of the basis is oriented as $\nabla q/ |\nabla q|$. With this choice,  
we notice that from \eqref{D2q} we know that the value of the element $(D^2 q)_{1,1}$ of the Hessian matrix  is 
$$(D^2 q)_{1,1}= \prt*{\frac{\nabla q}{|\nabla q|}}^{\!T} D^2 q \prt*{\frac{\nabla q}{|\nabla q|}}= -\frac{\alpha'}{2} \nabla q \cdot \nabla p.$$ 
Let us denote 
$$\lambda:=-\frac{\alpha'}{2}
\nabla q \cdot \nabla p, \quad \delta_{i}=(D^2q)_{i,i} \, \text{ for } i=2, \dots, d, \quad \delta := \sum_{i=2}^{d} \delta_i.$$
Since $\gamma p >0$, we have
\begin{align*}
    \mathcal{A} = \gamma \nabla q \cdot\nabla p \Delta q - \gamma p |D^2 q|^2 &\leq \gamma \nabla q \cdot\nabla p (\lambda + \delta) - \gamma p \prt*{\lambda^2 + \sum_{i=2}^d \delta_i^2}\\[0.5em]
    &\leq \gamma \nabla q \cdot\nabla p (\lambda + \delta) - \gamma p \prt*{\lambda^2 + \frac{\delta^2}{d-1}}\\[0.5em]
     &= - \gamma\frac{\alpha'}{2} |\nabla q \cdot\nabla p|^2 + \gamma\delta \nabla q \cdot\nabla p \\[0.2em] 
     &\qquad- \gamma p \frac{|\alpha'|^2}{4}|\nabla q \cdot\nabla p|^2-\gamma p \frac{\delta^2}{d-1}.
\end{align*}
Using Young's inequality, we have
\begin{equation*}
     \gamma\delta  \nabla q \cdot\nabla p \leq   \gamma p \frac{\delta^2}{d-1} + \gamma\frac{|\nabla q \cdot\nabla p|^2}{4p}(d-1) ,
\end{equation*}
and we finally find
\begin{align*}
    \mathcal{A} \leq \prt*{-\gamma \frac{\alpha'}{2} + \frac{\gamma(d-1)}{4p} - \gamma p \frac{|\alpha'|^2}{4}}|\nabla q \cdot\nabla p|^2.
\end{align*}
Coming back to the estimate of the time derivative, and using the above inequality, we obtain
\begin{equation}\label{final1} 
\begin{aligned}
\partialt{}\prt*{e^\alpha\frac{|\nabla q|^2}{2}} \leq e^\alpha &\left(|\nabla q|^2|\nabla p|^2 \prt*{-\frac{\alpha'}{2}+\gamma p \frac{|\alpha'|^2}{2} - \gamma p \frac{\alpha''}{2}} \right.\\[0.6em]
& \qquad+ |\nabla q \cdot\nabla p|^2 \prt*{- \gamma \frac{\alpha'}{2}+\frac{\gamma(d-1)}{4p} -\gamma p \frac{|\alpha'|^2}{4} }\\[0.6em]
&\qquad\left.+\gamma p \frac{\alpha'}{2}\Delta V |\nabla q|^2 - \nabla q D^2V \nabla q
\right).
    \end{aligned} 
\end{equation}
Now let us take $\alpha(p)= b \ln |p|$, with $b\in \R$ satisfying the assumptions of Proposition~\ref{prop: main}. In particular, $b \gamma>0$. Substituting in \eqref{final1} we get
\begin{equation}\label{final2} 
\begin{aligned}
\partialt{}\prt*{|p|^b\frac{|\nabla q|^2}{2}} \leq |p|^b &\left(|\nabla q|^2|\nabla p|^2 \prt*{-\frac{b}{2p}+\frac{\gamma b^2}{2p}  +\frac{\gamma b}{2p}}\right.\\[0.6em]
&\qquad\left. + |\nabla q \cdot\nabla p|^2 \prt*{- \frac{\gamma b}{2 p}+\frac{\gamma(d-1)}{4p} -\frac{\gamma b^2}{4p}}\right.\\[0.6em]
&\qquad\left.+\frac{\gamma b}{2}\Delta V |\nabla q|^2 - \nabla q D^2V \nabla q
\right).
    \end{aligned} 
\end{equation}
We denote
\begin{align*}
    {c_1}&=-\frac{|b|}{2}+\frac{|\gamma| b^2}{2} +\frac{\gamma |b|}{2}, \\
    {c_2}&=- \frac{\gamma |b|}{2}+\frac{|\gamma|(d-1)}{4} -\frac{|\gamma| b^2}{4},
    \end{align*}
    and
\begin{equation}\label{defic_0}
    c_0:=c_1+c_2 = -\frac{|b|}{2}+\frac{|\gamma| b^2}{4} +\frac{|\gamma|(d-1)}{4}.
\end{equation}
We rewrite \eqref{final2} as follows
\begin{equation}\label{final3}
\begin{aligned}
\partialt{}\prt*{|p|^b\frac{|\nabla q|^2}{2}}&\leq |p|^{b-1}\left(c_1 |\nabla q|^2|\nabla p|^2 + c_2|\nabla q \cdot\nabla p|^2\right)\\[0.5em]
&\quad+|p|^b\left(\frac{\gamma b}{2}\Delta V |\nabla q|^2 - \nabla q D^2V \nabla q\right).
\end{aligned} 
\end{equation}
Note that, because of $|\nabla q \cdot\nabla p|\leq |\nabla q||\nabla p|$, whenever we have $c_1\leq 0$, we can use
\begin{equation}\label{c1c2c0}
c_1 |\nabla q|^2|\nabla p|^2 + c_2|\nabla q \cdot\nabla p|^2\leq  c_0|\nabla q \cdot\nabla p|^2.
\end{equation}
Now we treat each different choice of potential separately.

\subsection{Generic potential with bounded derivatives} 
Let $V$ satisfy assumption~\eqref{assum: bdd V}, and let $n\in L^\infty((0,\infty)\times\R^d)$. By assumption \eqref{assum: coeff bdd V} - points \textit{(ii)} and \textit{(iii)} - we have $c_0<0, c_1\leq 0$, so that we can use \eqref{c1c2c0}. 

 Since $p=q-V$, from equation \eqref{final3} we find
\begin{equation*}
\begin{aligned}
\partialt{}\prt*{|p|^b\frac{|\nabla q|^2}{2}} &\leq c_0 |p|^{b-1} |\nabla q \cdot\nabla p|^2 +|p|^b\left(\frac{\gamma b}{2}\Delta V |\nabla q|^2 - \nabla q D^2V \nabla q\right)\\[0.6em]
&\leq c_0 |p|^{b-1}|\nabla q|^4 + c_0 |p|^{b-1} |\nabla q \cdot \nabla V|^2 - 2 c_0 |p|^{b-1}|\nabla q|^2 \nabla q \cdot \nabla V \\[0.2em]
&\qquad\quad +\frac{\gamma b}{2}|p|^b|\nabla q|^2 ||\Delta V||_\infty +|p|^b||D^2 V||_\infty|\nabla q|^2\\[0.6em]
&\leq c_0 |p|^{b-1} |\nabla q|^4 + 2 |c_0| |p|^{b-1}  |\nabla V|_\infty |\nabla q|^3\\[0.2em]
&\qquad\quad +C(\gamma,b,||D^2 V||_\infty)|p|^b|\nabla q|^2.
\end{aligned} 
\end{equation*}
By Young's inequality, we have
\begin{equation*}
    2|c_0||p|^{b-1}|\nabla V|_\infty |\nabla q|^3 \leq \frac 3 4 |c_0||p|^{b-1} |\nabla q|^4 + 4 |c_0||p|^{b-1}|\nabla V|_\infty^4,
\end{equation*}
hence
\begin{equation}\label{Vbddc_0C}
     \partialt{}\prt*{|p|^b\frac{|\nabla q|^2}{2}} \leq \frac{c_0}{4} |p|^{b-1} |\nabla q|^4 + 4|c_0| |p|^{b-1}||\nabla V||_{\infty}^4 + C |p|^b|\nabla q|^2,
\end{equation}
for some $C>0$.

Let us start from the case $\gamma>0$.
By the global boundedness of the density, there exists $\overline{p}>0$ such that $|p|\leq \overline{p}$. By assumption \eqref{assum: coeff bdd V} $b-1>0$, thus we have
\begin{equation*}
\begin{aligned}
    \partialt{}\prt*{|p|^b\frac{|\nabla q|^2}{2}} &\leq \frac{c_0}{4} |p|^{2b} |\overline{p}|^{-b-1}|\nabla q|^4 + 4|c_0||\overline{p}|^{b-1}||\nabla V||_{\infty}^4 + C |p|^b|\nabla q|^2 \\[0.6em]
    &\leq -C_1 |p|^{2b}|\nabla q|^4 + C_2 |p|^b |\nabla q|^2 + C,
    \end{aligned}
\end{equation*}
where $C_1, C_2$ denotes positive constants whose value may from now on change from line to line.
Recalling the definition $v(t):= \frac 12|p(t,x_0)|^b |\nabla q(t,x_0)|^2$, we obtain the differential inequality 
\begin{equation}\label{eq: ode 1}
v'(t_0)\leq - C_1 v^2(t_0) +C_2 v(t_0) + C.
\end{equation}
We now consider the fast diffusion case, $\gamma<0$. From $n \in L^\infty((0,\infty)\times\R^d)$, there exists $\overline{p}<0$ such that $|p(t,x)|\geq |\overline{p}|$ for all $(x,t)\in \R^d\times(0,\infty)$.
We come back to \eqref{Vbddc_0C} and write
\begin{align*} 
\partialt{}&\prt*{|p|^b\frac{|\nabla q|^2}{2}}\\[0.6em]
&\leq c_0|p|^{2b} |\overline{p}|^{-b-1}  |\nabla q|^4 + 2 |c_0| |p|^{3b/2} |\overline{p}|^{-b/2-1} |\nabla V|_\infty |\nabla q|^3   +  C|p|^b|\nabla q|^2, 
\end{align*}
where we used $-b-1\geq0$ and $-b/2 -1\leq 0$, by assumption \eqref{assum: coeff bdd V}. This can be written in the form
$$v'(t_0)\leq - C_1 v^2(t_0) + C v(t_0)^{3/2}+ C_2 v(t_0) + C,$$
and, after applying the Young inequality $v(t_0)^{3/2}\leq \varepsilon v(t_0)^{2}+C(\varepsilon)$ for suitable $\varepsilon>0$ small, we deduce again \eqref{eq: ode 1}.

\subsection{Quadratic potential} 
Under assumption \eqref{assum: attr V} we have $V=|x|^2/2$ and equation \eqref{final3} reads
\begin{equation*}
\begin{aligned}
\partialt{}\prt*{ |p|^b\frac{|\nabla q|^2}{2}} \leq |p|^{b-1} &\left(c_1 |\nabla q|^2|\nabla p|^2 + c_2|\nabla q \cdot\nabla p|^2\right) + c_3 |p|^b |\nabla q|^2,
    \end{aligned} 
\end{equation*}
where 
\begin{align*}
    c_3= \frac{\gamma b d}{2}-1.
\end{align*}
Conditions \textit{(ii)} and \textit{(iii)} of assumption \eqref{assum: coeff attr V} ensure $c_0\leq 0$, and $c_1, c_3 <0$. 
Therefore, we establish
\begin{equation*} 
\begin{aligned}
\partialt{}\prt*{|p|^b\frac{|\nabla q|^2}{2}} &\leq c_0 |p|^b|\nabla q \cdot\nabla p|^2 + c_3 |p|^b |\nabla q|^2 \\[0.6em]
&\leq  c_3 |p|^b |\nabla q|^2,
    \end{aligned} 
\end{equation*}
namely
\begin{equation}
    \label{eq: ode 2}
    v'(t_0) \leq - C v(t_0).
\end{equation}

\begin{remark}
Let us point out that the same conclusion holds for any potential $V$ that satisfies
$$ D^2V - \Delta V \frac{\gamma b}{2} I \geq c I,$$ for some positive constant $c=c(\gamma,b,d)$ where $I\in \R^d\times\R^d$ denotes the identity matrix. 
\end{remark}

\subsection{Trivial potential}\label{subsec: trivial}
If $V=0$, then $q=p$ and equation \eqref{final3} reads
\begin{equation}\label{fpmenoV}
    \partialt{}\left(|p|^b \frac{|\nabla p|^2}{2}\right) \leq c_0 |p|^{b-1} |\nabla p|^4.
\end{equation}
We recall that $c_0$ is defined as in \eqref{defic_0}, and is negative under assumption~\eqref{assum: coeff V=0}. 

Let us first prove the claimed result for $n_0\in L^\infty(\R^d)$. For $\gamma, b>0$, we have $|p(t,x)|^{-b-1}\geq |\overline{p}|^{-b-1}$. For $\gamma<0$,  by assumption $b\leq -1$, and $|p(t,x)|\geq |\overline{p}|$. Therefore, in both cases we find
\begin{align*}
      \partialt{}\left(|p|^b \frac{|\nabla p|^2}{2}\right) \leq c_0 |p|^{2b} |\overline{p}|^{-b-1} |\nabla p|^4,
\end{align*}
from which we find
\begin{equation}
    \label{eq: ode 3}
    v'(t_0)\leq - C v^2(t_0).
\end{equation}

\smallskip 

We now show that we may find a different differential inequality, which actually improves the rate of decay for large times, by using the bounds provided by Lemma~\ref{lemma: bound on n(t)}. For this, we no longer need the uniform boundedness assumption on the initial data.

For $\gamma>0$, we apply \eqref{eq: bound on n(t)} to \eqref{fpmenoV} to obtain
\begin{align*}
    \partialt{} \left(p^b \frac{|\nabla p|^2}{2}\right) &\leq c_0 p^{2b} p^{-(b+1)} |\nabla p|^4\\
    &\leq - C p^{2b} |\nabla p|^4 t_0^{(b+1)\frac{d \gamma}{d \gamma +2}},
\end{align*}
namely 
\begin{equation}
    \label{eq: ode 4}
    v'(t_0)\leq - C v^2(t_0) t_0^{\gamma d(b+1) \alpha}.
\end{equation}
In the fast diffusion case, we apply the lower bound of Lemma~\ref{lemma: bound on n(t)} to \eqref{fpmenoV} to infer
\begin{align*}
    \partialt{} \left(|p|^b \frac{|\nabla p|^2}{2}\right) &\leq c_0 |p|^{2b} |p|^{-(b+1)} |\nabla p|^4\\
    &\leq - C p^{2b} |\nabla p|^4 t_0^{(b+1)\frac{d \gamma}{d \gamma +2}},
\end{align*}
where we used $b+1<0$ by assumption \eqref{assum: coeff V=0}. Finally, we obtain again \eqref{eq: ode 4}.

\subsection{Proof of Proposition~\ref{prop: main}}
We conclude the proof of Proposition~\ref{prop: main} by showing that the differential inequalities \eqref{eq: ode 1},\eqref{eq: ode 2},\eqref{eq: ode 3} and \eqref{eq: ode 4} actually hold for the quantity $u(t)$
\begin{equation*} 
    u(t)= \max_x |p(t,x)|^b \frac{|\nabla q(t,x)|^2}{2}.
 \end{equation*} 
In order to do so, we first need to show that for every $t_0>0$ there exists $x_0$ such that \eqref{max} holds, namely the maximum is attained. Secondly, we will prove that $u=u(t)$ is locally Lipschitz on $(0,\infty)$. Before proving both claims in the next paragraph, let us draw now draw the main conclusions.

Since $u$ is locally Lipschitz it is also almost everywhere differentiable. Let $t_0>0$ be a point in which $u$ is differentiable. Since $u(t)\ge v(t)$ for all $t>0$ and $u(t_0)=v(t_0)$, we have $u'(t_0)=v'(t_0)$, and the differential inequalities proven for $v$ at $t=t_0$ also hold for the function $u$. Now, the definition of the function $v$ involved $x_0$ (and hence $t_0$), but this is no more the case for the function $u$. Hence, $u$ is a locally Lipschitz function on $(0,\infty)$ solving a.e. an ODE, and we can use such an ODE to obtain estimates on $u$. In particular, we may conclude that
\begin{itemize}
    \item for a generic potential $V$: \eqref{eq: ode 1} implies
    \begin{equation*} 
  u(t) \leq \max (C_1,C_2 t^{-1}), \quad \forall t>0,
\end{equation*}
and if in addition we assume $\sup_x |p(0,x)|^b |\nabla p(0,x)+ \nabla V(x)|^2<\infty$, this implies $u(0)<\infty$, hence
\begin{equation*}
  u(t) \leq C, \quad \forall t\geq 0,
\end{equation*}
\item for the quadratic potential: \eqref{eq: ode 2} gives
\begin{equation*}
    u(t)\leq C_0 e^{-C t}, \quad \forall t>0,
\end{equation*}
\item for the trivial potential: assuming the uniform boundedness of the initial data, from \eqref{eq: ode 3} we have
\begin{equation*} 
    u(t)\leq C t^{-1} \quad \forall t>0,
\end{equation*}
while from \eqref{eq: ode 4} we conclude
\begin{equation}\label{rate pme}
    u(t) \leq C t^{-1-\gamma d (b+1)\alpha}.
\end{equation}
\end{itemize}
This last rate of decay is actually sharp in that equality holds for the Barenblatt solution \eqref{Barenblatt density pme}. The pressure of the Barenblatt profile is given by \eqref{Barenblatt pressure} and its gradient is
\begin{equation*}
   \nabla \curlyP(t,x)=- \sign(\gamma)\alpha \frac{x}{t}\mathds{1}_{\{ \alpha |x|^2 \leq 2 C t^{2\alpha}\}}.
\end{equation*}
Let $t_0>0$ and $x \in \{x; C t_0^{2\alpha} \leq \alpha |x|^2 \leq \frac 32 C t_0^{2\alpha}\}$. In this region the solution is positive and $x$ is far from the origin. Therefore, we have
\begin{equation*}
 |x|\simeq t_0^{\alpha}, \qquad |\curlyP|\simeq t_0^{-\gamma\alpha  d}, \qquad |\nabla \curlyP|\simeq t_0^{-1+\alpha},
\end{equation*}
hence
\begin{equation*}
    |\curlyP|^b |\nabla \curlyP|^2 \simeq t_0^{-\gamma \alpha d b-2 +2\alpha},
\end{equation*}
which is in fact the same exponent of equation \eqref{rate pme}.

Finally, we now show that the maximum is attained and that $u(t)$ is locally Lipschitz. The argument is different for the PME and FDE.

\smallskip
\noindent
\textbf{Porous medium equation.}
Since the solution satisfies  Definition~\ref{definition}, the pressure is smooth inside its support, globally Lipschitz, and vanishes on the free boundary. Hence, the quantity $|p(t_0,x)|^b |\nabla p(t_0,x)+ V(x)|^2$ attains its maximum value in a point $x_0\in \R^d$ inside the support. 

It now remains to show that the function $u=u(t)$ is locally Lipschitz in $(0,\infty)$. 
Take an instant $t_0>0$. Unless $n(t_0)$ is a stationary solution to the equation (in which case there is essentially nothing to prove), we have $u(t_0)>0$.
By assumption $n$ is continuous in $(t,x)$ so that a simple semicontinuity argument shows that there exists $\delta>0$ such that
$$u(t)\geq \frac{u(t_0)}{2}, \quad \forall t \in [t_0 -\delta, t_0 + \delta].$$
Let us denote $L:=\sup_{t\in [t_0-\delta, t_0+\delta]}  \mathrm{Lip}(q(t))$.
For $t\in[t_0-\delta, t_0+\delta]$, we can restrict the set on which we take the maximum in the definition of the function $u(t)$. More precisely, we have for $t$ in this interval 
$$u(t) = \max_{x\in H(t)} |p(t,x)|^b \frac{|\nabla q(t,x)|^2}{2},$$
where $H(t):=\{x; \ |p(t,x)|^b\geq \frac{u(t_0)}{4L^2}\}.$ It is now sufficient to prove that there exists a set $E\subset \R^d$, independent of time, such that $H(t)\subset E \subset \{x; \ p(t,x)>0\}$ for all $t\in [t_0-\delta, t_0 +\delta]$.
Let us define 
$$E:=\left\{x; \ p(t_0,x)\geq a\right\}, \quad \text{with } a:= \left(\frac{u(t_0)}{8 L^2}\right)^{1/b}.$$
Then, let us suppose that there exist $t_n \in [t_0-\delta, t_0+\delta]$, and $x_n\in \R^d$ such that $p(t_n, x_n) \geq 2 a$ and $p(t_0, x_n)< a$, namely, that $x_n \in H(t_n)$ and $x_n \notin E$. Since there exists a bounded set $K\subset \R^d$ large enough such that $x_n \in \supp{p(t_n)}\subset K$, we can extract a subsequence $x_{n_k}$ such that $x_{n_k}\to x \in K$, and therefore 
$$p(t_0,x)\geq 2a, \qquad p(t_0,x)<a,$$
which is a contradiction. By a similar argument, one can show that $E \subset \{x; \ p(t,x)>0\}$.
Therefore, we have proven that we have
$$u(t)= \max_E |p(t,x)|^b \frac{|\nabla p(t,x)+\nabla V(x)|^2}{2},$$
where $E$ is independent of time. Since $|p(t,x)|^b |\nabla p(t,x)|^2$ is $C^\infty$ on the positivity set of the pressure, it is, therefore, Lipschitz continuous in $t$ uniformly in $x$. Hence $u(t)$ is Lipschitz continuous on $[t_0-\delta, t_0 + \delta]$, and this concludes the proof.  

\smallskip
\noindent
\textbf{Fast diffusion equation.} 
We now show that there exists $x_0\in \R^d$ in which the quantity $|p(t_0)|^b|\nabla q(t_0) |^2$ attains its maximum. To this end, we show $\lim_{|x|\to\infty}|p|^b|\nabla q|^2=~0$. This tail behavior follows from assumption \eqref{cond: tail behavior} and the fact that we imposed $b>-1$.  
Since $n(t,x)$ satisfies \eqref{cond: tail behavior}, we have
\begin{equation*}
     |\nabla p(t,x)|\leq C(t)(1+|x|), 
\end{equation*}
where the constant $C(t)$ is locally bounded in time. Note moreover that in any of the three cases (trivial potential, quadratic potential, generic potential with bounded derivatives) we also have
$|\nabla V(x)|\leq C(1+|x|)$.
Therefore, we get
$$|p(t,x)|^b |\nabla q(t,x)|^2 \leq C(t)(1+|x|^2)^{b+1},$$
and we obtain the desired result since $b<-1$, which proves that the maximum is attained. 

In order to conclude that $u$ is differentiable almost everywhere and $u'(t_0)=v'(t_0)$, it remains to show that the function $u=u(t)$ is locally Lipschitz in $(0,\infty)$ and exactly as for the porous medium case, we only need to show that the maximum can be localy restricted to a fixed subset independent of time. Again it is enough to assume $u(t_0)>0$, find $\delta>0$ such that $u(t)>u(t_0)/2$ for $t\in [t_0-\delta,t_0+\delta]$, and restrict to a ball $B(0,R)$ such that 
$ |p(t,x)|^b \frac{|\nabla q(t,x)|^2}{2}<\frac{u(t_0)}{2}$ for $t\in [t_0-\delta, t_0+\delta]$ and $x\notin B(0,R)$, which is possible because the above estimates on $p$ and $\nabla p$ are supposed to be uniform in time.

\subsection{Proof of the main results}
To extend the results for $V=0$ and $V=|x|^2/2$ to the class of solutions with $L^1$ initial data, we proceed by approximation and compactness arguments. In order to show that the approximating sequence is compact and that its initial data convergences to $n_0$, we need to show local equicontinuity in time. To this end, it is crucial to infer a bound on $n^{\gamma+1}$.
\begin{lemma}\label{lemma: bound for equic}
    Let $n(t,x)$ be the solution of equation~\eqref{eq: pme}, with $|\gamma|<1$. There exists $r>1$ such that $n^{\gamma+1}\in L^r([0,T]\times\R^d))$ for every $T>0$. 
\end{lemma}
\begin{proof}
 For any $T>0$, we have
\begin{align*}
    \int_0^T \int_{\R^d} |n(t)|^{(\gamma+1)r} \dx x \dx t &\leq \int_0^T \|n(t)\|_\infty^{(\gamma+1) r-1} \prt*{\int_{\R^d} n(t) \dx x} \dx t\\[0.5em]
    &= M \int_0^T \|n(t)\|_\infty^{(\gamma+1) r-1}\dx t.
\end{align*}
By Lemma~\ref{lemma: bound on n(t)} we have
$$\|n(t)\|_\infty\leq Ct^{-\frac{d}{d\gamma+2}},$$
which shows that it is enough to choose $r>1$ such that 
$$[(\gamma+1)r-1]\frac{d}{d\gamma+2}<1,$$
\ie $1<r<\frac{\gamma+2/d+1}{\gamma+1}.$ 
\end{proof}
We note that in the case $\gamma<0$ the same result could be obtained in a much simpler way, since we have $n\in L^1$ and $\frac{1}{\gamma+1}>1$.

We have now all the elements to prove the main results.
\begin{proof}[Proof of Theorem~\ref{thm: V=0}] 
If $\gamma>0$, let us take $n_{0,\varepsilon}\in C_c(\R^d)$ such that its support is a ball $B(0,R_\varepsilon)$ and $n_{0,\varepsilon} \rightarrow n_0$ strongly in $L^1(\R^d)$. Moreover, let us assume that the initial pressure $p_{0,\varepsilon}= (n_{0,\varepsilon})^\gamma$ is such that $|\nabla p_{0,\varepsilon}(x)|$ is bounded from below by a positive constant on $\partial B(0,R_\varepsilon)$ and that $n_{0,\varepsilon}$ is strictly positive inside $B(0,R_\varepsilon)$. It is known that the solution $n_{\varepsilon}(t,x)$ of \eqref{eq: density} with initial data $n_{0,\varepsilon}$ satisfies Definition~\ref{definition}. 

If $\gamma<0$, let us take $n_{0,\varepsilon}\in \mathcal{X}\setminus \{0\}$. Then by \cite[Theorem 1.1]{BS22} and  \cite[Theorem 4]{Blanchet2009}, the solution $n_{\varepsilon}(t,x)$ of equation~\eqref{eq: pme} satisfies property \eqref{cond: tail behavior}, and therefore Definition~\ref{definition}.
We conclude that Proposition~\ref{prop: main} holds for $n_{\varepsilon}$. 

Because of the well-known contractivity in $L^1$ of the PME an FDE we deduce that $n_{\varepsilon}(t)$ is Cauchy in $L^1(\R^d)$ for every $t$ and we call $n$ its limit as $\varepsilon\to 0$.

Thanks to Proposition~\ref{prop: main}, given $\tau>0$, $\big|\nabla n_\varepsilon^{\gamma(b/2+1)}(t)\big|$ is uniformly bounded for all $t>\tau$. The control of this quantity implies that the sequence of regularized solutions is equicontinuous in space. Equicontinuity in time follows, for instance, from Lemma~\ref{lemma: bound for equic}, namely the fact that $\partial_t n_\varepsilon$ is uniformly bounded in $L_{\mathrm{loc}}^{r}(0,\infty;W^{-2,r}(\R^d))$. Thus, for all $\tau>0$ we have
$$n_\varepsilon \rightarrow n \text{ in } C_{\mathrm{loc}}((\tau,\infty)\times \R^d),$$
and $n(t,x)$ satisfies equation~\eqref{eq: pme} for all $t>0$.
Moreover, $n(t)\to n_0$ as $t\to 0$ thanks to the uniform bound on $n$ in $W^{1,r}(0,T;W^{-2,r}(\R^d))$. Therefore, $n(t,x)$ is the unique solution of equation~\eqref{eq: pme} with initial data $n_0$.
Since $n \mapsto \max_x |\nabla n^{\gamma(b/2+1)}|^2$ is lower semi-continuous we conclude.
\end{proof}

\begin{proof}[Proof of Theorem~\ref{thm: attr V}] 
We argue by approximation as before. In fact, through the time-scaling \eqref{eq: change of v}, equation~\eqref{eq: density} with $V=|x|^2/2$ is equivalent to equation \eqref{eq: pme}.
The equicontinuity properties still hold since they are not affected by the change of variable, and therefore we have uniform convergence of the sequence. The functional $p \mapsto \max_x |p|^b |\nabla p + x|^2$ is lower semi-continuous, and therefore we conclude.  
\end{proof}

\subsection{Proof of Corollaries~\ref{cor: sqrt} and \ref{cor: lip n}}
We now discuss some interesting implications of our result in the trivial potential case, for $b\to -1$ and $b=2/\gamma -2$.

\begin{proof}[Proof of Corollary~\ref{cor: sqrt}]
    Let us take $n_0\in L^1(\R^d)\cap L^\infty(\R^d)$. From Theorem~\ref{thm: V=0} we have
    \begin{align*}
        \max_x |p(t)|^b |\nabla p(t)|^2 \leq C t^{-1},
    \end{align*}
    for all $b<-1$ that satisfy assumption \eqref{assum: coeff V=0}. From the computation carried out in Section~\ref{subsec: trivial} and in particular from inequality~\eqref{eq: ode 3}
we find 
$$C=\frac{1}{2|c_0||\bar{p}|^{-b-1}},$$
with $c_0$ defined as in \eqref{defic_0} and $\bar p:=||n_0||_{L^\infty}^\gamma$.
Let us now take the limit $b\to -1$. Since in this case we have $2|c_0|\to 1/(2\alpha)$, we obtain
 \begin{align*}
        \max_x |p(t)|^b |\nabla p(t)|^2 \leq 2 \alpha t^{-1}.
    \end{align*}
This estimate has been obtained assuming $n_0\in L^\infty$, but it no longer depends on the $L^\infty$ norm of $n_0$. Hence, by approximation, the result is also true for any given initial data $n_0\in L^1(\R^d)$ (actually the assumption $n_0\in L^1$ can also be removed, and is kept only in order to provide a precise functional meaning to the equation).

Computing the Lipschitz norm of $\sqrt{\curlyP}$, where the profile of the Barenblatt pressure is given by equation~\eqref{Barenblatt pressure} with $\gamma<0$, we find
\begin{equation*}
    \max_x \frac{|\nabla \curlyP(t,x)|^2}{|\curlyP(t,x)|}= 2\alpha t^{-1},
\end{equation*}
which indicates that the constant is sharp.
\end{proof}

\begin{proof}[Proof of Corollary~\ref{cor: lip n}]
The result is a straightforward application of Theorem~\ref{thm: V=0} with $\gamma b=2(1-\gamma)$. Let us point out that for $0<\gamma\le 4/(d+3)$ this choice of $b$ satisfies assumption~\eqref{assum: coeff V=0}. In particular, from \eqref{defic_0} we have
\begin{equation*}
    c_0= \frac{\gamma}{4}(d+3)-1\le 0.
\end{equation*}
  This provides the first part of the claim. Moreover, if the inequality $\gamma< 4/(d+3)$ is strict, we also obtain a decay estimate which provides, from Theorem~\ref{thm: V=0},
  \begin{equation*}
     \max_x |\nabla n(t,x)|^2 \leq C t^{-1-\alpha d (2-\gamma)}.\qedhere
  \end{equation*}    
\end{proof}

\section{Asymptotic behaviour at large times}\label{sec: asymptotic results}
We want to study the asymptotic behavior as $t\to\infty$ of the solution $n(t,x)$ to the drift-less porous medium and fast diffusion equations \eqref{eq: pme}.
In order to do so, we take advantage of the fact that $n(t,x)$ can be seen as the solution of a convective porous medium equation with quadratic potential through the change of variables \eqref{eq: change of v}. 
\begin{proof}[Proof of Theorem \ref{thm: convergence grad}]
Since $\hat n$ defined in \eqref{eq: change of v} satisfies equation \eqref{eq hat n}, by Theorem~\ref{thm: attr V} we have
\begin{equation*}
   \max_x |\hn(t,x)|^{\gamma b} \ \big|\sign(\gamma)\nabla \hn^\gamma(t,x) + x\big|^2 \leq C_0 e^{-C t}.
\end{equation*}
Using the definition of $\hn(t,x)$ we find
\begin{align*}
 &e^{\gamma b d t} \max_x  |n(\psi(t), e^t x)|^{\gamma b} \ \left||\gamma| e^{d(\gamma-1)t}e^{(d+1)t}n^{\gamma-1}(\psi(t), e^t x) \nabla n(\psi(t), e^t x) + x \right|^2 \\[0.8em]
 &\qquad\qquad\qquad\leq C_0 e^{-C t}.
\end{align*}
Let us denote $s:=\psi(t)$ and $y=e^t x$. Substituting in the equation above, we obtain sequentially
\begin{equation*}
\begin{aligned}
s^{\frac{\gamma b d}{d\gamma+2}} \max_y  |n(s,y)|^{\gamma b} \ \left||\gamma| s^{\frac{d\gamma +1}{d\gamma+2}} n^{\gamma-1}(s,y) \nabla n(s,y)   + y s^{-\frac{1}{d\gamma +2}} \right|^2 &\leq C_0 s^{-\frac{C}{d\gamma +2}},\\[0.8em]
s^{\frac{\gamma b d}{d\gamma+2}} \max_y   |n(s,y)|^{\gamma b} \ \left|\sign(\gamma) s^{\frac{d\gamma +1}{d\gamma+2}}\nabla |n(s,y)|^\gamma  + y s^{-\frac{1}{d\gamma +2}} \right|^2 &\leq C_0 s^{-\frac{C}{d\gamma +2}},\\[0.8em]
 s^{\frac{\gamma d (2+b) +2}{d\gamma+2}}  \max_y |n(s,y)|^{\gamma b} \ \left|\sign(\gamma) \nabla n^\gamma(s,y)  + y s^{-1} \right|^2 &\leq C_0 s^{-\frac{C}{d\gamma +2}}.
\end{aligned}
\end{equation*}
Therefore
\begin{equation}
    \label{eq: final rate}
\max_x |n(t,x)|^{\gamma b} \ \left|\sign(\gamma) \nabla n^\gamma(t,x)  + x t^{-1} \right|^2 \leq C_0 t^{\beta}t^{-\frac{C}{d\gamma +2}},
\end{equation}
with $\beta = -\alpha\gamma d b -2 + 2\alpha$, and this concludes the proof.
\end{proof}

\smallskip
\noindent
\textbf{Weighted convergence of the pressure gradient.}
For $-2/d<\gamma<0$, since $n(t,x)$ is always positive and the gradient of the Barenblatt pressure is proportional to $x t^{-1}$, equation \eqref{eq: final rate} gives an interesting insight on the rate of convergence of the gradient of $n^{\gamma}$ for large times. Let us take $n_0\in \mathcal{X}\setminus\{0\}$. We recall that the pressure of the Barenblatt profile of the fast diffusion equation is
\begin{equation*}
    \curlyP(t,x) = t^{-\alpha\gamma d}\prt*{C - \alpha\frac{|x|^2}{2}  t^{-2\alpha}}.
\end{equation*} 
Let us take $x\in C_{a,b}(t):=\{x; \ |x|t^{-\alpha}\in [a,b] \}$ for some constants $0<a<b$. Since by \cite[Theorem 1.1]{BS22} $n$ satisfies \eqref{cond: btween B} if and only if $n_0\in \mathcal{X}\setminus \{0\}$, we have, for $x\in C_{a,b}(t)$, a behavior of the form $n^{\gamma b}(t,x)\gtrsim t^{-\alpha \gamma d b}$. Moreover, in such region we also have $|\nabla \curlyP(t,x)|^2 \simeq t^{-2+2\alpha}$. Let us recall $\beta= -\alpha \gamma d b -2 +2\alpha$. Finally, from \eqref{eq: final rate} we find
\begin{equation*} 
\dfrac{\big\|\nabla p(t)  - \nabla \curlyP(t)\big\|^2_{L^\infty(C_{a,b}(t))}}{\big\|\nabla \curlyP(t)\big\|_{L^\infty(C_{a,b}(t))}^2}\lesssim t^{-\frac{C}{d\gamma +2}}, \quad \forall t>0.
\end{equation*}

\section{Problem posed on a bounded domain}\label{sec: bounded}
We now discuss the same problem posed on bounded domains, both for Neumann and Dirichlet homogeneous boundary conditions. For an overview of the asymptotic behavior of the PME and FDE in bounded domains, we refer the reader to \cite{BF21, BFV18, V04} and references therein.

\subsection{Neumann boundary conditions on a convex domain}
Let us now consider the same equation on a bounded, smooth, and convex domain $\Omega\subset \R^d$
\begin{equation*}
   \left\{\begin{array}{rll} 
   \partialt{n} &\!\!= \nabla \cdot (n \nabla q),  &\text{in } (0,\infty)\times \Omega,\\[0.7em]
   n(0,x)&\!\!=n_0(x), &\text{in } \Omega,\\[0.7em]
   \nabla q \cdot \nu &\!\!=0,  &\text{on } (0,\infty)\times \partial\Omega,\\[0.7em]
   \end{array}\right.
\end{equation*}
where $q= p +V$ and $\nu$ represents the outward normal to $\partial \Omega$. Let us assume that $$\partial_\nu V \ge 0  \text{ on } \partial\Omega.$$ 
The computation performed in Section~\ref{sec: proof main} follows through in the same way. We only need to ensure that the maximum of the quantity $e^{\alpha(p)}{|\nabla q|^2}$ is not attained on $\partial\Omega$.

First of all, we denote $h(x)$ a convex function such that 
$\Omega = \{x; h(x)<0\}$, $\partial\Omega=\{x; h(x)=0\}$. We assume $\nabla h\neq 0$ and $D^2h>0$ (in the sense of positive-definite symmetric matrices) on $\partial\Omega$. We fix an instant $t$ and write $q$ for $q(t,\cdot)$. Therefore, for any curve $\omega:(0,1)\to \partial \Omega$ we have $\nabla q(\omega(s)) \cdot \nabla h(\omega(s))=0.$ We differentiate (in $s$) and obtain
$$\omega'(s) (D^2(q(\omega(s)))  \nabla h (\omega(s)) + D^2 h(\omega(s))\nabla q(\omega(s)))=0.$$
We can choose $\omega'(s)=\nabla q(\omega(s))$ (which is a possible choice, since $\nabla q$ is a tangent vector) to obtain
\begin{equation}\label{eq: contrad}
\nabla q(\omega(s)) D^2(q(\omega(s))) \nabla h (\omega(s)) + \nabla q(\omega(s)) D^2 h(\omega(s))\nabla q(\omega(s)))=0.
\end{equation}
Let us assume that the maximum is attained on the boundary. This implies that there exists a positive constant $\mu$ such that
$$\alpha'(p) {|\nabla q(\omega(s))|^2} \nabla p(\omega(s)) + 2 D^2q(\omega(s)) \nabla q(\omega(s)) = \mu \nu,$$
which, taking the scalar product with $\nu= \nabla h$, gives
\begin{align*}\nabla h(\omega(s)) D^2q(\omega(s)) \nabla q(\omega(s)) >0,
\end{align*}
since $\partial_\nu p = \partial_\nu q - \partial_\nu V \leq 0$ on $\partial \Omega$ and $\alpha'>0$ (we recall that we use $\alpha(p)=b\log|p|$, so that $\alpha'(p)=b/p$ and the choice of the sign of $b$ guarantees that $b$ and $p$ always have the same sign). Thus, from \eqref{eq: contrad} we deduce
\begin{align*} 
\nabla q(\omega(s)) D^2 h(\omega(s))\nabla q(\omega(s))) <0,
\end{align*}
which is a contradiction since $h$ is convex.

\subsection{Dirichlet boundary conditions}
    For the problem set in a bounded domain with homogeneous Dirichlet boundary conditions, the computations of Section~\ref{sec: proof main} could still be performed at a formal level. However, the quantity under investigation, namely
    \begin{equation*}
        u(t)= \max_x |p(t,x)|^b |\nabla p(t,x)|^2 =  \max_x |\nabla n^{\gamma(b/2+1)}(t,x)|^2,
    \end{equation*}
is actually not finite for our range of exponents $b$. Indeed, let us consider the drift-less problem
\begin{equation*}
      \left\{\begin{array}{rll} 
   \partialt{n} &\!\!= \dfrac{|\gamma|}{\gamma+1}\Delta n^{\gamma+1},  &\text{in } (0,\infty)\times \Omega\\[0.7em]
   n(0,x)&\!\!=n_0(x), &\text{in } \Omega,\\[0.7em]
    n &\!\!=0,  &\text{on } (0,\infty)\times \partial\Omega.\\[0.7em]
   \end{array}\right.
\end{equation*}
This problem admits a solution of the form $n(t,x)= a(x) b(t)$, where the functions $a(x), b(t)$ satisfy
\begin{equation*} 
        b'(t)=- \dfrac{|\gamma|}{\gamma+1} b^{\gamma+1}(t), \ \text{ and } \
        a(x)=-\Delta a^{\gamma+1}(x), 
\end{equation*}
for $t>0$ and $x \in \Omega$.
We denote $\Tilde{a}(x):= a^{\gamma+1}(x)$. This function satisfies $-\Delta \Tilde{a}=\Tilde{a}^{1/(\gamma+1)}>0,$ in $\Omega$ and $\Tilde{a}=0$ on $\partial \Omega$. Thus, on the boundary, we have $\partial_\nu \Tilde{a}<0$. Let us now consider the quantity \[|\nabla a^{\gamma(b/2+1)}|=|\nabla \Tilde{a}^{\theta}|,\] 
with $\theta=\gamma(b/2+1)/(\gamma+1)$. Even for the maximum value of $b$ allowed, the one such that we have $\gamma b= 1 +\sqrt{1-\gamma^2(d-1)},$ we obtain $\theta<1$. Therefore the quantity $|\nabla a^{\gamma(b/2+1)}|$ blows up on the boundary, hence $u(t)=+\infty$.

\section*{Acknowledgments}
The authors would like to thank Matteo Bonforte, Philippe Lauren\c{c}ot, Nikita Simonov, and Juan Luis V\'azquez for fruitful discussions during the preparation of this paper.
This project was supported by the LABEX MILYON (ANR-10-LABX-0070) of Université de Lyon, within the program «Investissements d’Avenir» (ANR-11-IDEX-0007) operated by the French National Research Agency (ANR), and by the European Union via the ERC AdG 101054420 EYAWKAJKOS project.
The authors also acknowledge the support of the Lagrange Mathematics and Computation Research Center, which also hosted important preliminary discussions on this topic, via its project on Optimal Transportation.

\bibliographystyle{abbrv}
\bibliography{biblio}

\begin{thebibliography}{10}

\bibitem{aronson70}
D.~G. Aronson.
\newblock Regularity properties of flows through porous media: The interface.
\newblock {\em Archive for Rational Mechanics and Analysis}, 1970.

\bibitem{AB79}
D.~G. Aronson and P.~B\'enilan.
\newblock R\'{e}gularit\'{e} des solutions de l'\'{e}quation des milieux poreux
  dans {${\bf R}^{N}$}.
\newblock {\em C. R. Acad. Sci. Paris S\'{e}r. A-B}, 288(2):A103--A105, 1979.

\bibitem{AGV98}
{Aronson, D.G., Gil, O., V\'azquez, J.L.}
\newblock Limit behaviour of focusing solutions to nonlinear diffusions.
\newblock {\em Comm. Partial Differential Equations}, 23(1-2):307--332, 1998.

\bibitem{BGIL16}
S.~Benachour, R.~G. Iagar, and P.~Laurençot.
\newblock Large time behavior for the fast diffusion equation with critical
  absorption.
\newblock {\em Journal of Differential Equations}, 260(11):8000--8024, 2016.

\bibitem{Benilannotes}
P.~B\'enilan.
\newblock Evolution equations and accretive operators.
\newblock {\em Lecture Notes, Univ. Kentucky, manuscript}, 1981.

\bibitem{Blanchet2009}
A.~Blanchet, M.~Bonforte, J.~Dolbeault, G.~Grillo, and J.~L. Vázquez.
\newblock Asymptotics of the fast diffusion equation via entropy estimates.
\newblock {\em Arch. Rational Mech. Anal.}, 191:347–385, 2009.

\bibitem{BDNS22}
M.~Bonforte, J.~Dolbeault, B.~Nazaret, and N.~Simonov.
\newblock Stability in gagliardo-nirenberg-sobolev inequalities: flows,
  regularity and the entropy method.
\newblock {\em Arxiv Preprint: arXiv:2007.03674}, 2022.

\bibitem{BF21}
M.~Bonforte and A.~Figalli.
\newblock Sharp extinction rates for fast diffusion equations on generic
  bounded domains.
\newblock {\em Comm. Pure Appl. Math.}, 74(4):744--789, 2021.

\bibitem{BFV18}
M.~Bonforte, A.~Figalli, and J.~L. V\'azquez.
\newblock Sharp global estimates for local and nonlocal porous medium-type
  equations in bounded domains.
\newblock {\em Anal. PDE}, 11(4):945–982, 2018.

\bibitem{BS22}
M.~Bonforte and N.~Simonov.
\newblock Fine properties of solutions to the cauchy problem for a fast
  diffusion equation with {C}affarelli–{K}ohn–{N}irenbergweights.
\newblock {\em Ann. Inst. H. Poincaré C Anal. Non Linéaire}, 40(1):1--59,
  2022.

\bibitem{BV06}
M.~Bonforte and J.~L. V{\'a}zquez.
\newblock Global positivity estimates and harnack inequalities for the fast
  diffusion equation,.
\newblock {\em Journal of Functional Analysis}, 240:399--428, 2006.

\bibitem{CF80}
L.~A. Caffarelli and A.~Friedman.
\newblock Regularity of the free boundary of a gas flow in an n-dimensional
  porous medium.
\newblock {\em Indiana University Mathematics Journal}, 29(3):361–391, 1980.

\bibitem{CVW87}
L.~A. Caffarelli, J.-L. Vázquez, and N.~Wolanski.
\newblock Lipschitz continuity of solutions and interfaces of the
  n–dimensional porous medium equation.
\newblock {\em Indiana University Mathematics Journal}, 36(2):373–401, 1987.

\bibitem{C01}
J.~A. Carrillo, A.~J\"ungel, P.~A. Markowich, G.~Toscani, and A.~Unterreiter.
\newblock Entropy dissipation methods for degenerate parabolic systems and
  generalized sobolev inequalities.
\newblock {\em Monatsh. Math.}, 133:1--82, 2001.

\bibitem{CT00}
J.~A. Carrillo and G.~Toscani.
\newblock Asymptotic l1-decay of solutions of the porous medium equation to
  self-similarity.
\newblock {\em Indiana Univ. Math. J.}, 49:113--141, 2000.

\bibitem{CV2003}
J.~A. Carrillo and J.~L. Vázquez.
\newblock Fine asymptotics for fast diffusion equations.
\newblock {\em Communications in Partial Differential Equations},
  28:1023–1056, 2003.

\bibitem{CJM15}
H.~{Hajj Chehade}, M.~Jazar, and R.~Monneau.
\newblock A priori gradient bounds for fully nonlinear parabolic equations and
  applications to porous medium models.
\newblock {\em Journal de Mathématiques Pures et Appliquées},
  103(6):1346--1357, 2015.

\bibitem{HQZ17}
Y.~Hu, Z.~Qian, and Z.~Zhang.
\newblock Gradient estimates for porous medium and fast diffusion equations by
  martingale method.
\newblock {\em Ann. Inst. H. Poincaré Probab. Statist.}, 53(4):1793–1820,
  2017.

\bibitem{kalashnikov}
A.~S. Kalashnikov.
\newblock On the occurrence of singularities in the solutions of the equation
  of nonstationary filtration.
\newblock {\em Z. Vych. Mat. i. Mat. Fisiki}, 7:440--444, 1967.

\bibitem{KMcC2006}
Y.~J. Kim and R.~J. McCann.
\newblock Potential theory and optimal convergence rates in fast nonlinear
  diffusion.
\newblock {\em Journal de Mathématiques Pures et Appliquées}, 86:42--67,
  2006.

\bibitem{LV03}
K.~Lee and J.~L. V{\'a}zquez.
\newblock Geometrical properties of solutions of the porous medium equation for
  large times.
\newblock {\em Indiana Univ. Math. J.}, 52(4):991–1016, 2003.

\bibitem{OKJ58}
O.~A. Oleinik, A.~S. Kalashnikov, and J.~lin Czou.
\newblock The {C}auchy problem and boundary problems for equations of the type
  of non-stationary filtration.
\newblock {\em Izv. Akad. Nauk SSSR Ser. Mat.}, 22(5):667–704, 1958.

\bibitem{Ot01}
F.~Otto.
\newblock The geometry of dissipative evolution equations: the porous medium
  equation.
\newblock {\em Comm. Partial Differential Equations}, 26:101–174, 2001.

\bibitem{V83}
J.~L. V{\'a}zquez.
\newblock Asymptotic behavior and propagation properties of the one-dimensional
  flow of a gas in a porous medium.
\newblock {\em Trans. AMS}, 277:507--527, 1983.

\bibitem{V2003}
J.~L. V{\'a}zquez.
\newblock {\em Asymptotic behaviour for the porous medium equation posed in the
  whole space}, pages 67--118.
\newblock Birkh{\"a}user Basel, 2003.

\bibitem{V04}
J.~L. V\'azquez.
\newblock The {D}irichlet problem for the porous medium equation in bounded
  domains. asymptotic behaviour.
\newblock {\em Monatsh. Math.}, 142:81–111, 2004.

\bibitem{V}
J.~L. Vazquez.
\newblock {\em The porous medium equation: mathematical theory}.
\newblock Oxford Mathematical Monographs. The Clarendon Press, Oxford
  University Press, Oxford, 2007.
\newblock Mathematical theory.

\bibitem{V06}
J.~L. Vázquez.
\newblock {\em {Smoothing and Decay Estimates for Nonlinear Diffusion
  Equations: Equations of Porous Medium Type}}.
\newblock Oxford University Press, 2006.

\end{thebibliography}

\end{document}